\documentclass{amsart}
\usepackage{amssymb,latexsym,enumerate,verbatim,url,mathabx}
\newtheorem{Thm}{Theorem}[section]
\newtheorem{Prop}[Thm]{Proposition}
\newtheorem{Cor}[Thm]{Corollary}
\newtheorem{Lem}[Thm]{Lemma}

\theoremstyle{remark}

\theoremstyle{problem}

\numberwithin{equation}{section}
\begin{document}

\title[$p$-restriction and $p$-spectral synthesis]
{Sets of $p$-restriction and $p$-spectral synthesis}

\author[M. J. Puls ]{Michael J. Puls }
\address{Department of Mathematics \\
John Jay College-CUNY \\
524 West 59th Street \\
New York, NY 10019 \\
USA}
\email{mpuls@jjay.cuny.edu}

\begin{abstract}
In this paper we investigate the restriction problem. More precisely, we give sufficient conditions for the failure of a set $E$ in $\mathbb{R}^n$ to have the $p$-restriction property. We also extend the concept of spectral synthesis to $L^p(\mathbb{R}^n)$ for sets of $p$-restriction when $p > 1$. We use our results to show that there are $p$-values for which the unit sphere is a set of $p$-spectral synthesis in $\mathbb{R}^n$ when $n \geq 3$.  
\end{abstract}


\date{March 30, 2019}
\maketitle

\section{Introduction}\label{Introduction}
Throughout this paper $\mathbb{R}$ will denote the real numbers and $\mathbb{Z}$ will denote the integers. Let $p \in [1, \infty]$ and let $n$ be a positive integer. Indicate by $L^p(\mathbb{R}^n)$ the usual Lebesgue space and denote by $\Vert \cdot \Vert_{L^p}$ the usual Banach space norm on $L^p(\mathbb{R}^n)$. If $X \subset L^p(X)$, then $\overline{X}^p$ will denote the closure of $X$ in $L^p(\mathbb{R}^n)$. Also $p'$ will always represent the conjugate index of $p$, that is $\frac{1}{p} + \frac{1}{p'} =1$. For $f \in L^1(\mathbb{R}^n)$ the Fourier transform of $f$ is defined by
\[ \hat{f}(\xi) = \mathcal{F}(f)(\xi) = \int_{\mathbb{R}^n} e^{-2\pi i x \cdot \xi}f(x) dx, \xi \in \mathbb{R}^n. \]
The Fourier transform can be extended to a unitary operator on $L^2(\mathbb{R}^n)$ and by the Hausdorff-Young inequality, $ \mathcal{F}$ can be extended to a continuous operator from $L^p(\mathbb{R}^n)$ to $L^{p'}(\mathbb{R}^n)$, when $1 < p <2$. Let $E$ be a closed subset of $\mathbb{R}^n$. If $f \in L^1(\mathbb{R}^n)$, then $\hat{f}$ is continuous on $\mathbb{R}^n$. Consequently, the restriction of $\hat{f}$ to $E$, which we denote by $\hat{f}\mid_E$, is a well-defined function on $E$. For $1 < p < 2$ and $f \in L^p(\mathbb{R}^n), \hat{f} \in L^{p'}(\mathbb{R}^n)$. So if $E$ is a set of positive Lebesgue measure we can restrict $\hat{f}$ to $E$. The interesting question is can $\hat{f}$ be restricted to $E$ when $E$ has Lebesgue measure zero? This question is the heart of the restriction problem, which we will now describe.

For $E \subseteq \mathbb{R}^n$, let $C(E)$ be the set of continuous functions on $E$ and let $C_c(E)$ be the set of functions in $C(E)$ with compact support. Let $L^p(E)$ be the usual Banach space formed with respect to the induced measure $d\sigma$ on $E$. The norm on $L^p(E)$ will be denoted by $\Vert \cdot \Vert_{L^p(E)}$. Recall that the norm on $L^p(\mathbb{R}^n)$ is indicated by $\Vert \cdot \Vert_{L^p}$. Let $2 \leq n \in \mathbb{Z}$ and let $\mathcal{S}(\mathbb{R}^n)$ denote the space of Schwartz functions on $\mathbb{R}^n$. The operator $\mathcal{R}_E\colon \mathcal{S}(\mathbb{R}^n) \rightarrow C(E)$ given by 
\[ \mathcal{R}_E (f) = \hat{f}\mid_E \]
is known as the {\em restriction operator} associated with $E$. If $\mathcal{R}_E$ can be extended to a continuous operator from $L^p(\mathbb{R}^n) \rightarrow L^q(E)$, then we shall say that $\mathcal{R}_E$ has property $\mathcal{R}(E, p, q)$. Observe that if $1 \leq q_1 \leq q_2$ and $\mathcal{R}_E$ has property $\mathcal{R}(E, p, q_2)$, then it also has property $\mathcal{R}(E, p, q_1)$. We shall say that $E$ is a set of {\em $p$-restriction} if $\mathcal{R}_E$ has property $\mathcal{R}(E, p, 1)$. Note that any closed set in $\mathbb{R}^n$ is a set of $1$-restriction. Furthermore, if $E$ is not a set of $p$-restriction, then $\mathcal{R}_E$ does not have property $\mathcal{R}(E, p, q)$ for any $q \geq 1$. The best known result concerning the restriction of $\hat{f}$ to $E$ is the Stein-Tomas theorem: $\mathcal{R}_E$ has property $\mathcal{R}(E, p, 2)$ if and only if $1 \leq p \leq \frac{2n+2}{n+3}$, where $E$ is a smooth compact hypersurface in $\mathbb{R}^n$ with nonzero Gaussian curvature. In general though it is an extremely difficult problem to determine if $\mathcal{R}_E$ has property $\mathcal{R}(E, p, q)$. A more comprehensive treatment of the restriction problem, along with its history, can be found in \cite{FoschiOliviera17, Tao04}, and \cite[Chapter 5.4]{Grafakos09} and the references therein. 

In this paper we will only be concerned with the case where $\mathcal{R}_E$ has property $\mathcal{R}(E, p, 1)$, that is $E$ is a set of $p$-restriction. Set
\[ J(E) = \{ f \in \mathcal{S}(\mathbb{R}^n) \mid \hat{f}\mid_E = 0 \}, \]
and if $E$ is a set of $p$-restriction define
\[ I^p(E) = \{f \in L^p(\mathbb{R}^n) \mid \hat{f}\mid_E = 0 \}. \]
This paper was inspired by the paper \cite{Stolyarov16} where these spaces were investigated for the case when $E$ is the unit sphere $S^{n-1}$ in $\mathbb{R}^n$. Our first main result is: 
\begin{Thm} \label{notrestrictionlowerdim} Let $E$ be a smooth compact submanifold of codimension $k$ in $\mathbb{R}^n$. If there exists $f \in C_c(\mathbb{R}^n)$ such that $\hat{f}$ vanishes on $E$, then $E$ is not a set of $p$-restriction for $p > \frac{2n}{n+k}$.
\end{Thm}
Thus $\mathcal{R}_E$ does not have property $\mathcal{R}(E, p, q)$ when $p > \frac{2n}{n+k}$ and $q \geq 1$. If $E$ is a hypersurface, then the lower bound for $p$ becomes $\frac{2n}{n+1}$. We are able to improve this lower bound for hypersurfaces with the constant relative nullity condition, which we now define. Let $U$ be an open set in $\mathbb{R}^{n-1}$ and let $F = \{ (x, \phi(x)) \mid x \in U \}$ be a smooth hypersurface in $\mathbb{R}^n$. If the Hessian matrix
\[ \left( \frac{\partial^2 \phi}{\partial x_i \partial x_j}\right) \]
of $\phi$ has constant rank $n-1 - \nu$ on $U$, where $0 \leq \nu \leq n-1$, then we say that $\phi$ has {\em constant relative nullity $\nu$}. A smooth hypersurface $E$ of $\mathbb{R}^n$ is said to have {\em constant relative nullity $\nu$} if every localization $F$ of $E$ has constant relative nullity $\nu$. If $\nu = n-1$, then $E$ is a hyperplane. It is known that hyperplanes are sets of $p$-restriction only if $p=1$. Thus we will only consider hypersurfaces with $0 \leq \nu \leq n-2$. Note that $\nu = 0$ for $S^{n-1}$.
\begin{Thm} \label{nullitynotres}
Let $2 \leq n \in \mathbb{Z}$ and let $E$ be a smooth compact hypersurface in $\mathbb{R}^n$ with constant relative nullity $\nu$, for $0 \leq \nu \leq n-2$. If $\frac{2(n -\nu)}{n - \nu +1} < p \in \mathbb{R}$, then $E$ is not a set of $p$-restriction.
\end{Thm}

Let $f \in L^p(\mathbb{R}^n)$ and let $y \in \mathbb{R}^n$. The translate of $f$ by $y$, which we write as $f_y$, is the function $f_y(x) = f(x-y)$, where $x \in \mathbb{R}^n$. For $f \in L^p(\mathbb{R}^n)$, let $T^p[f]$ be the closed subspace of $L^p(\mathbb{R}^n)$ spanned by $f$ and its translates. The zero set $Z(f)$ of $ f \in L^1(\mathbb{R}^n)$ is defined by
\[ Z(f) = \{ \xi \in \mathbb{R}^n \mid \hat{f}(\xi) = 0 \}. \]
In Section \ref{proofsnotrestrictthm} we will see that if $Z(f)$ is a set of $p$-restriction, then $T^p[f] \neq L^p(\mathbb{R}^n)$.

We will now briefly review the concept of spectral synthesis in $L^1(\mathbb{R}^n)$. Suppose $I$ is a closed ideal in $L^1(\mathbb{R}^n)$ and define the zero set of $I$ by
\[ Z(I) = \bigcap_{f \in I} Z(f). \]
Let $E$ be a closed set in $\mathbb{R}^n$, then $I^1(E)$ is a closed ideal in $L^1(\mathbb{R}^n)$ with zero set $E$. In fact, $I^1(E)$ is the largest closed ideal in $L^1(\mathbb{R}^n)$ whose zero set is $E$. Now let
\[ k(E) = \{ f \in \mathcal{S}(\mathbb{R}^n) \mid \hat{f} =0 \mbox{ on a neighborhood of } E\}. \]
Then
\[ k(E) \subseteq J(E) \subseteq L^p(\mathbb{R}^n) \]
and $\overline{k(E)}^1$ is the smallest closed ideal in $L^1(\mathbb{R}^n)$ with zero set $E$. The set $E$ is known as a set of spectral synthesis if $\overline{k(E)}^1 = I^1(E)$. A more detailed account of spectral synthesis can be found in \cite{Domar71}\cite[Chapter 7]{Rudin90}. Extending the concept of spectral synthesis to $L^p(\mathbb{R}^n)$ for $p >1$ falls short since the analog to $I^1(E), I^p(E)$, is not well-defined for closed sets of Lebesgue measure zero in $\mathbb{R}^n$. However, for sets of $p$-restriction $I^p(E)$ is well-defined, which allows us to extend the idea of spectral synthesis to $L^p(\mathbb{R}^n)$ for sets $E$ of $p$-restriction.  We shall say that a set $E$ of $p$-restriction is a set of {\em $p$-spectral synthesis} if 
\[ \overline{k(E)}^p = I^p(E). \]
We can now state:
\begin{Thm} \label{psyn}
Let $2 \leq n \in \mathbb{Z}$ and let $E$ be a smooth compact hypersurface in $\mathbb{R}^n$ with constant relative nullity $\nu, 0 \leq \nu \leq n-2$. If $E$ is a set of $p$-restriction for some $p$ that satisfies one of the following:
\begin{enumerate}
\item $\frac{2(n - \nu)}{n+3 - \nu} \leq p < 2$ and $0 \leq \nu  < n-3$
\item $1 < p < 2$ for $n-3 \leq \nu < n-1$,
\end{enumerate}
then $E$ is a set of $p$-spectral synthesis.
\end{Thm}

It is known that $S^1$ is a set of spectral synthesis in $\mathbb{R}^2$ \cite{Herz58}, but  $ S^{n-1}$ is not a set  of spectral synthesis in $\mathbb{R}^n$ for $n \geq 3$ \cite[Chapter 7.3]{Rudin90}. We will use Theorem \ref{psyn} to show that there are $p$-values  where $S^{n-1}$ is a set of $p$-spectral synthesis for $n \geq 3$. 

This paper is organized as follows: In Section \ref{Preliminaries} we give some background and results that will be needed for this paper. In Section \ref{proofsnotrestrictthm} we will prove Theorems \ref{notrestrictionlowerdim} and \ref{nullitynotres} by linking them to the problem of determining when $T^p[f]$ is dense in $L^p(\mathbb{R}^n)$ for $f \in \mathcal{S}(\mathbb{R}^n)$ with $\hat{f} = 0$ on $E$. In Section \ref{pspectralsyn} we prove Theorem \ref{psyn}, and use the theorem to show that there are $p$-values for which the unit sphere $S^{n-1}$ is a set of $p$-spectral synthesis in $\mathbb{R}^n$ for $n \geq 3$. 

\section{Preliminaries}\label{Preliminaries}
In this section we will give some results that will be used in the sequel. The convolution of two measurable functions $f$ and $g$ on $\mathbb{R}^n$ is defined by
\[ f \ast g(x) = \int_{\mathbb{R}^n} f(x-y) g(y) \, dy.\]
Let $1<p<2$. Each $\phi \in L^{p'}(\mathbb{R}^n)$ defines a bounded linear functional $T_{\phi}$ on $L^p(\mathbb{R}^n)$ via
\[ T_{\phi} (f) = \int_{\mathbb{R}^n} f(-x)\phi(x)\, dx. \]
Sometimes we will write $\langle f, \phi \rangle$ in place of $T_{\phi}(f)$. For closed subspaces $X$ in $L^p(\mathbb{R}^n)$,
\[ \text{Ann}(X) = \{ \phi \in L^{p'}(\mathbb{R}^n) \mid T_{\phi} (f) =0 \text{ for all } f\in X\}, \]
will denote the annihilator of $X$ in $L^{p'}(\mathbb{R}^n)$. The following characterization of $\text{Ann}(X)$ when $X$ is a translation-invariant subspace of $L^p(\mathbb{R}^n)$ will be needed later.
\begin{Prop} \label{annihilator} Let $X$ be a translation-invariant subspace of $L^p(\mathbb{R}^n)$. Then $\phi \in \text{Ann}(X)$ if and only if $f \ast \phi = 0$ for all $f \in X$.
\end{Prop}
\begin{proof}
Observe that for $f \in L^p(\mathbb{R}^n)$ and $\phi \in L^{p'}(\mathbb{R}^n)$
\[ f \ast \phi (x) = \int_{\mathbb{R}^n} f(x-y) \phi(y)\, dy = \int_{\mathbb{R}^n} f_{-x}(-y) \phi(y) \, dy = T_{\phi} (f_{-x}). \]
It follows from the translation invariance of $X$ that $f \ast \phi = 0$ for all  $f\in X$ if and only if $\phi \in \text{Ann}(X)$.
\end{proof}
The space $L^p(\mathbb{R}^n)$ is a $L^1(\mathbb{R}^n)$-module since $f \ast g \in L^p(\mathbb{R}^n)$ whenever $f \in L^1(\mathbb{R}^n)$ and $g \in L^p(\mathbb{R}^n)$. The following proposition will not be used in the paper, but we record it here for its independent interest.
\begin{Prop} \label{submodule}
If $E$ is a set of $p$-restriction, then $I^p(E)$ is a $L^1(\mathbb{R}^n)$-submodule of $L^p(\mathbb{R}^n)$.
\end{Prop}
\begin{proof}
A modification of the proof of \cite[Theorem 7.1.2]{Rudin90} will show that a closed translation-invariant subspace of $L^p(\mathbb{R}^n)$ is translation invariant if and only if it is a $L^1(\mathbb{R}^n)$-submodule of $L^p(\mathbb{R}^n)$. The proposition now follows since $I^p(E)$ is a closed translation-invariant subspace of $L^p(\mathbb{R}^n)$.
\end{proof}

It is well known that the Fourier transform is an isomorphism on the Schwartz space $\mathcal{S}(\mathbb{R}^n)$, with inverse Fourier transform given by
\[ \widecheck{f}(x)  = \int_{\mathbb{R}^n} f (\xi) e^{2\pi i(\xi \cdot x)}\, d\xi \]
for $f \in \mathcal{S}(\mathbb{R}^n)$. A continuous linear functional on $\mathcal{S}(\mathbb{R}^n)$ is known as a {\em temperate distribution}. A nice property of temperate distributions is that the Fourier transform can be extended to them. In fact, the Fourier transform defines an isomorphism on the temperate distributions. Indeed, if $T$ is a tempered distribution, then $\widehat{T}$ is the tempered distribution given by
\[ \widehat{T} (f) = T(\widehat{f}) \]
for $f \in \mathcal{S}(\mathbb{R}^n)$. The inverse Fourier transform $\widecheck{T}$ of a temperate distribution $T$ is defined by
\[ \widecheck{T}(f) = T(\widecheck{f}), \]
where $f \in \mathcal{S}(\mathbb{R}^n)$. Since elements of $L^p(\mathbb{R}^n)$ are temperate distributions, we can define the Fourier transform $\widehat{f}$ for $f \in L^p(\mathbb{R}^n)$ in the distributional sense when $p >2$. {\em For the rest of this paper, distribution will mean temperate distribution}.

We shall write $\text{supp}(\psi)$ to indicate the support of $\psi$, where depending on the context, $\psi$ is a function, measure, or distribution.

We conclude this section with a result that will be needed later.
\begin{Prop}\label{vanishingf}
If $E$ is a compact subset of $\mathbb{R}^n$, then there exists an $f \in \mathcal{S}(\mathbb{R}^n)$ for which $Z(f) =E$.
\end{Prop}
\begin{proof}
Let $B$ be an open ball containing $E$ and let $x \in B \setminus E$. The Whitney extension theorem produces a smooth function $f_x \colon B \rightarrow \mathbb{R}$ such that $f_x =0$ on $E$ and $f_x >0$ at $x$. For the purpose of this proof only, $f_x$ will mean the function defined above instead of the translate of $f$. For each $x \in B \setminus E$ there exists an open ball $B_x$ for which $f_x$ is positive on $B_x$. Now choose a countable subcover $B_{x_n}$ of $B \setminus E$. Let
\[ a_n = n^{-2}[\sup_B (f_{x_n})]^{-1}. \]
Then
\[ g = \sum_{n=1}^{\infty} a_n f_{x_n} \]
is a smooth function on $B$. Let $B_1$ be an open ball satisfying $E \subseteq B_1 \subseteq B.$ Deonte by $h$ the smooth function obtained by multiplying $g$ by a smooth function that equals one on $B_1$ and zero on $\mathbb{R}^n\setminus B$. Set $F = h+s$ where $s \in \mathcal{S}(\mathbb{R}^n)$ that is zero on $B_1$ and positive on $\mathbb{R}^n\setminus \overline{B_1}$, where $\overline{B_1}$ is the closure of $B_1$. Thus $F \in \mathcal{S}(\mathbb{R}^n)$ and can be expressed as $\widehat{f}$ for some $f \in \mathcal{S}(\mathbb{R}^n)$. The proof of the proposition is now complete since $F^{-1}(0) = E$.
\end{proof}

\section{Proofs of Theorems \ref{notrestrictionlowerdim} and \ref{nullitynotres}}\label{proofsnotrestrictthm}
Let $E$ be a compact set in $\mathbb{R}^n$ with induced measure $d\sigma$. Suppose $E$ has the $p$-restriction property. This is equivalent to the existence of a constant $C$ that depends on $p$ and $n$ and satisfies 
\begin{equation}\label{eq:contprestrict}
\Vert \hat{f} \Vert_{L^1(E)} \leq C \Vert f \Vert_{L^p}
\end{equation}
for all $f \in \mathcal{S}(\mathbb{R}^n)$. Condition (\ref{eq:contprestrict}) is equivalent to the dual condition
\begin{equation}\label{eq:adjcontext}
\Vert \widecheck{Fd\sigma}\Vert_{L^{p'}} \leq C \Vert F \Vert_{L^{\infty}(E)}
\end{equation}
for all smooth functions $F$ on $E$, and where $\widecheck{Fd\sigma}$ is the inverse Fourier transform of the measure $Fd\sigma$. Recall that the inverse Fourier transform of a finite Borel measure is $d\mu$
\[ \widecheck{d\mu}(x) = \int_{\mathbb{R}^n} e^{2\pi i (x \cdot \xi)} d\mu(\xi), \]
where $x \in \mathbb{R}^n$. Setting $F \equiv 1$ on $E$ we see from (\ref{eq:adjcontext}) that $\widecheck{d\sigma} \in L^{p'}(\mathbb{R}^n)$. We record this as:
\begin{Lem}\label{measureindual}
Let $1 < p <2$ and let $E$ be a compact set in $\mathbb{R}^n$ with induced measure $d\sigma$. If $E$ is a set of $p$-restriction, then $\widecheck{d\sigma} \in L^{p'}(\mathbb{R}^n)$.
\end{Lem}
\begin{Prop} \label{partialWiener}
Let $1 \leq p < 2$ and let $E$ be a compact subset of $\mathbb{R}^n$ with induced measure $d\sigma$. If $E$ is a set of $p$-restriction, then $I^p(E) \neq L^p(E)$.
\end{Prop}
\begin{proof}
Let $f \in I^p(E)$ and let $\phi(x) = \widecheck{d\sigma}(x)$. By Lemma \ref{measureindual}, $\phi(x) \in L^{p'}(\mathbb{R}^n)$. Let $f \in I^p(E)$ and let $(f_n)$ be a sequence of Schwartz functions that satisfy $\Vert f_n - f\Vert_{L^p} \rightarrow 0$. For $x \in \mathbb{R}^n$,
\[\vert f \ast \phi (x)\vert  = \lim_{n \rightarrow \infty}\vert f_n \ast \phi (x)\vert = \lim_{n \rightarrow \infty} \left\vert\int_{\mathbb{R}^n} e^{2\pi i(x \cdot \xi)} \widehat{f_n}(\xi) d\sigma(\xi)\right\vert. \]
Because $\widehat{f}\mid_E =0$ we obtain
\begin{eqnarray*}
 \lim_{n \rightarrow \infty} \left\vert \int_E e^{2\pi i (x \cdot \xi)} \widehat{f_n}(\xi)d\sigma(\xi)\right\vert & \leq & \lim_{n \rightarrow \infty} \int_E \vert (\widehat{f} - \widehat{f_n})(\xi) e^{2\pi i (x \cdot \xi)}\vert d\sigma(\xi) \\
                                                                                                                                                            & \leq & \lim_{n \rightarrow \infty}\Vert \widehat{f} - \widehat{f_n} \Vert_{L^1(E)}.
\end{eqnarray*}
Since  $\mathcal{R}_E$ has property $\mathcal{R}(E, p, 1), \lim_{n \rightarrow \infty}\Vert \widehat{f} - \widehat{f_n} \Vert_{L^1(E)} = 0$. Hence,  $f \ast \phi = 0$  and $\phi$ is a nonzero element in $\text{Ann}(I^p(E))$ by Proposition \ref{annihilator}. 
Thus $I^p(E) \neq L^p(\mathbb{R}^n)$. 
\end{proof}
The following corollary to Proposition \ref{partialWiener}, which is crucial for the proofs of Theorems \ref{notrestrictionlowerdim} and \ref{nullitynotres}, gives a useful criterion in terms of $T^p[f]$ to determine when $E$ is not a set of $p$-restriction.
\begin{Cor} \label{corpartialWiener}
Let $1 \leq p < 2$ and let $E$ be a compact subset of $\mathbb{R}^n$. If there exists an $f \in L^1(\mathbb{R}^n) \bigcap L^p(\mathbb{R}^n)$ for which $E \subseteq Z(f)$ and $T^p[f] = L^p(\mathbb{R}^n)$, then $E$ is not a set of $p$-restriction. 
\end{Cor}
\begin{proof}
Assume $E$ is a set of $p$-restriction. Then by Proposition \ref{partialWiener} $I^p(E) \neq L^p(\mathbb{R}^n)$. Since $Z(f_y) = Z(f)$ for all $y \in \mathbb{R}^n, T^p[f] \subseteq I^p(E)$ which contradicts our hypothesis $T^p[f] = L^p(\mathbb{R}^n)$.
\end{proof}
\subsection{Proof of Theorem \ref{notrestrictionlowerdim}}\label{prooflowerdim}
Theorem \ref{notrestrictionlowerdim} follows immediately by combining Corollary \ref{corpartialWiener} with \cite[Corollary 1]{Agran_Nara04}.
\subsection{Proof of Theorem \ref{nullitynotres}} \label{notnullityres}
We will prove Theorem \ref{nullitynotres} by proving a more general theorem. We start with a definition. Let $E$ be a closed subset of $\mathbb{R}^n$. We shall say that $E$ is {\em $p$-thin} if the only distribution $T$ that satisfies $\hbox{supp }T \subseteq E$ and $\widecheck{T} \in L^p(\mathbb{R}^n)$ is $T=0$. 
\begin{Thm}\label{pthinpres}
Let $1 < p < 2$ and let $E$ be a compact subset of $\mathbb{R}^n$. If $E$ is $p'$-thin, then $E$ is not a set of $p$-restriction.
\end{Thm}
\begin{proof}
Assume that $E$ is $p'$-thin. Let $f \in \mathcal{S}(\mathbb{R}^n)$ with $Z(f) = E$. The theorem will follow from Corollary \ref{corpartialWiener} if we can show $T^p[f] = L^p(\mathbb{R}^n)$. Suppose instead that $T^p[f] \neq L^p(\mathbb{R}^n)$. By Proposition \ref{annihilator} there exists a nonzero $\phi \in L^{p'}(\mathbb{R}^n)$ for which $f \ast \phi = 0$, which implies $\text{supp }\hat{\phi} \subseteq Z(f)$ because $\widehat{f \ast \phi} = \hat{f}\hat{\phi}$. Due to our assumption $E$ is $p'$-thin, $\phi = 0$, a contradiction. Hence $T^p[f]  = L^p(\mathbb{R}^n)$.
\end{proof}
It was shown in \cite[Theorem 1]{Guo93} that if a set $E$ satisfies the hypothesis of Theorem \ref{nullitynotres} then $E$ is $p'$-thin. Therefore, $E$ is not a set of $p$-restriction and Theorem \ref{nullitynotres} is proved.

\section{$p$-spectral synthesis} \label{pspectralsyn}
We start with a definition. Let $E$ be a $k$-dimensional submanifold in $\mathbb{R}^n$ with induced Lebesgue measure $d\sigma$. We shall say that $E$ has the {\em $p$-approximate property} if for each distribution $T$ with $\text{supp }T \subseteq E$ and $\widecheck{T} \in L^p(\mathbb{R}^n)$, we can find a sequence of measures $T_j$ on $E$, absolutely continuous with respect to $d\sigma$, such that $\Vert \widecheck{T_j} - \widecheck{T} \Vert_{L^p} \rightarrow 0$ as $j \rightarrow \infty$. Our results on sets of $p$-spectral synthesis are an immediate consequence of previous work by Guo on sets that have the $p$-approximate property \cite{Guo89, Guo95}. In fact, it is stated in \cite{Guo89} that the $p$-approximate property is a variation of the spectral synthesis property. Sets with the $p$-restriction property allows us to make this statement more transparent. Specifically, Theorem \ref{papproximppspec} will show that $p$-spectral synthesis follows from the $p'$-approximate property for submanifolds with the $p$-restriction property. 

\subsection{Proof of Theorem \ref{psyn}}
Theorem \ref{psyn} will follow immediately by combining \cite[Theorem 1]{Guo89} and \cite[Theorem 2]{Guo95} with Theorem \ref{papproximppspec} below.

Suppose $E$ is a $k$-dimensional submanifold of $\mathbb{R}^n$ and let $d\sigma$ be the induced Lebesgue measure on $E$. Also assume that $E$ is a set of $p$-restriction for some $1 < p < 2$. Let $\Phi$ denote the closed subspace of $L^{p'}(\mathbb{R}^n)$ generated by 
\[ \{ \widecheck{Fd\sigma} \mid F \text{ is smooth on } E \}. \]
\begin{Thm} \label{papproximppspec}
Let $1 < p < 2$ and let $E$ be a compact, smooth $k$-dimensional submanifold of $\mathbb{R}^n$ and assume that $E$ has the $p$-restriction property. Let $d\sigma$ be the induced measure on $E$. If $E$ has the $p'$-approximate property, then $E$ is a set of $p$-spectral synthesis.
\end{Thm}
\begin{proof}
Since $E$ is a set of $p$-restriction, $\widecheck{d\sigma} \in L^{p'}(\mathbb{R}^n)$ by Lemma \ref{measureindual}. We begin by showing showing $\Phi \subseteq \text{Ann}(I^p(E))$. Let $\phi \in \Phi$. We can assume that $\phi = \widecheck{d\mu}$, where $d\mu = F d\sigma$ for some smooth function $F$ on $E$. Let $f \in I^p(E)$, using the argument from Proposition \ref{partialWiener} we obtain that $f \ast \phi = 0$, which implies that  $\phi \in \text{Ann}(I^p(E))$ by Proposition \ref{annihilator}. 

Now let $\phi \in \text{Ann}(\overline{k(E)}^p)$. Since $\text{supp}(\widehat{\phi}) \subseteq E$ and $E$ has the $p'$-approximate property, there exists a sequence of measures $F_n d\sigma$, where $F_n$ is smooth on $E$, such that $\Vert \widecheck{F_n d\sigma} - \phi \Vert_{L^{p'}} \rightarrow 0$. Thus $\phi \in \Phi$, which implies $\text{Ann}(\overline{k(E)}^p) \subseteq \text{Ann}(I^p(E))$.  Clearly, $\text{Ann}(I^p(E)) \subseteq \text{Ann}(\overline{k(E)}^p)$. Therefore, $E$ is a set of $p$-spectral synthesis. 
\end{proof}
\subsection{$p$-spectral synthesis and the unit sphere}
We mentioned in the Introduction that for $n \geq 3, S^{n-1}$ is not a set of spectral synthesis. Using Theorem \ref{psyn} we will be able to show that there are $p$-values for which $S^{n-1}$ is a set of $p$-spectral synthesis. By \cite[Theorem 1]{Guo89}, $S^2$ has the $p'$-approximate property when $p' > 2$ and $n=3$; and $S^{n-1}$ has the $p'$-approximate property for $2 < p' \leq \frac{2n}{n-3}$ when $n \geq 4$. It follows from the Stein-Tomas theorem that $S^{n-1}$ is a set of $p$-restriction for $1 \leq p \leq \frac{2n+2}{n+3}$. Consequently, Theorem \ref{papproximppspec} yields that $S^2$ is a set of $p$-spectral synthesis for $1 < p \leq \frac{4}{3}$ and for $n \geq 4, S^{n-1}$ is a set of $p$-spectral synthesis for $\frac{2n}{n+3} \leq p \leq \frac{2n+2}{n+3}$. The upper bound in this inequality is probably not sharp, in fact it would not surprise us if it is $\frac{2n}{n+1}$. However, the lower bound is sharp. Indeed, using \cite[Lemma 2.3(ii)]{Guo93}  a distribution $\phi$ can be constructed with $\phi \in L^{p'}(\mathbb{R}^n) \mbox{for } p' > \frac{2n}{n-3}, \text{supp}(\widehat{\phi}) \subseteq S^{n-1}$ that satisfies for $f \in \mathcal{S}(\mathbb{R}^n)$, 
\[ \langle f, \phi \rangle = \int_{S^{n-1}} \frac{\partial \widehat{f}}{\partial x_n} \, d\sigma. \]
Since there exists $f \in \mathcal{S}(\mathbb{R}^n)$ with $\widehat{f} \mid_{S^{n-1}} = 0$ and $\frac{\partial \widehat{f} }{\partial x_n} \mid_{S^{n-1}} \neq 0, \phi \notin \text{Ann}(I^p(E))$. Thus, for $n \geq 4, S^{n-1}$ is not a set of $p$-spectral synthesis for $1 \leq p < \frac{2n}{n+3}$.

\bibliographystyle{plain}
\bibliography{prestrictpspec}

\end{document}